\documentclass[a4paper,10pt]{amsart}

\usepackage[english]{babel}

\usepackage{lmodern}

\usepackage[utf8]{inputenc}

\usepackage[babel=true]{csquotes}

\usepackage{comment}

\usepackage{amsfonts,amssymb,amsmath,eucal,pinlabel,array,hhline}

\usepackage{amsfonts}
\usepackage{amsthm} 
\usepackage{amsmath}
\usepackage{amssymb}

\DeclareMathOperator\artanh{artanh}
\DeclareMathOperator\arsinh{arsinh}

\usepackage{hyperref}

\theoremstyle{definition}
\newtheorem{defi}{Definition}[section] 

\theoremstyle{plain}
 
\newtheorem{prop}[defi]{Proposition}
\newtheorem{theo}[defi]{Theorem}
\newtheorem{coro}[defi]{Corollary}

\theoremstyle{remark}
\newtheorem{rema}[defi]{Remark}
\newtheorem{exem}[defi]{Example}
\newtheorem{question}[defi]{Question}

\newcommand{\Id}{\mathrm{Id}}

\newcommand{\C}{\mathbb{C}}
\newcommand{\R}{\mathbb{R}}

\newcommand{\Z}{\mathbb{Z}}
\newcommand{\N}{\mathbb{N}}
\newcommand{\F}{\mathbb{F}}
\newcommand{\NN}{\mathcal{N}}

\title[Fuglede--Kadison determinants over free groups \& Lehmer's constants]{Fuglede--Kadison determinants over free groups and Lehmer's constants}
\author{Fathi Ben Aribi}

\address{
	UCLouvain, IRMP, Chemin du Cyclotron 2 \\
	1348 Louvain-la-Neuve \\
	Belgium}
\email{fathi.benaribi@uclouvain.be}

\makeatletter
\@namedef{subjclassname@2020}{%
	\textup{2020} Mathematics Subject Classification}
\makeatother


\begin{document}

\begin{abstract}Lehmer's famous problem asks whether the set of Mahler measures of polynomials with integer coefficients admits a gap at $1$. In 2019, Lück extended this question to Fuglede--Kadison determinants of a general group, and he defined the Lehmer's constants of the group to measure such a gap.

	In this paper, we compute new values for Fuglede-Kadison determinants over non-cyclic free groups, which yields
	the new upper bound $\frac{2}{\sqrt{3}}$ for Lehmer's constants of all torsion-free groups which have non-cyclic free subgroups.
	
	Our proofs use relations between Fuglede--Kadison determinants and random walks on Cayley graphs, as well as works of Bartholdi and Dasbach--Lalin.
	
	Furthermore, via the gluing formula for $L^2$-torsions, we show that
	the Lehmer's constants of an infinite number of fundamental groups of hyperbolic 3-manifolds are bounded above by even smaller values than $\frac{2}{\sqrt{3}}$.
\end{abstract}

\maketitle

\section*{Introduction}

The \textit{Mahler measure} $\mathcal{M}(P) \geqslant 0$ of a polynomial $P(X) \in \C[X]$, given as the  geometric mean of $P$ over the unit circle, is a widely studied object in number theory. In particular, the famous
\textit{Lehmer's problem} asks if $1$ is an accumulation point of Mahler measures of polynomials with integer coefficients. If, on the contrary, there should exist a gap around $1$, one popular candidate for the lowest such Mahler measure greater than $1$ is $\mathcal{M}(L)=1.176280818...$, the Mahler measure of Lehmer's polynomial 
$L(X)$.

Fuglede--Kadison determinants are analytic variants of the usual determinant for operators over (possibly infinite-dimensional) Hilbert spaces such as the completion $\ell^2(G)$ of a group algebra. Fuglede--Kadison determinants are technical to define and notoriously difficult to compute. The most common exact computations of a Fuglede--Kadison determinant are either $1$ or  of the form $\exp \left (\frac{\mathrm{vol}(M)}{6\pi}\right )$ where $\mathrm{vol}(M)$ is a hyperbolic volume (via Theorem  \ref{thm:LS}). In this paper, we expand this range of explicit computations in the case of free groups (see Theorem \ref{thm:intro:det}).

Certain operators over $\ell^2(\Z)$ can be identified with elements of the ring of Laurent polynomials $\C[\Z] \cong \C[X^{\pm 1}]$, and their Fuglede--Kadison determinants are then given by the appropriate Mahler measure. In 2019, this led Lück \cite{Lu2} to generalise Lehmer's problem to other groups than $G=\Z$ and study accumulation points of Fuglede--Kadison determinants. In particular, Lück defined the four \textit{Lehmer's constants} $\Lambda(G), \Lambda^w(G), \Lambda_1(G), \Lambda^w_1(G)$ of a group $G$, that are greater than $1$ exactly when there is a gap around $1$ for certain Fuglede--Kadison determinants with integer coefficients (see Section \ref{sub:lehmer:luck} for more details).

In \cite{Lu2}, Lück presented several open questions,  among which the following one:

\begin{question}[Question \ref{qu:lehmer:luck} (2)]\label{qu:intro:lehmer:luck} 
	For which torsion-free groups $G$  do we have
	$$  \Lambda(G) = \Lambda^w(G) = \Lambda_1(G)= \Lambda^w_1(G) = \mathcal{M}(L)=1.17628... \ ?$$
\end{question}

A partial negative answer to Question \ref{qu:intro:lehmer:luck} was given by Lück (see Example \ref{ex:weeks}) for two Lehmer's constants of  the fundamental group $G_{We}$ of the Weeks manifold; for this group, Lück proved that
$\Lambda(G_{We}) \leqslant \Lambda^w(G_{We}) < \mathcal{M}(L)=1.17628...$

In this paper, we go further and answer Question \ref{qu:intro:lehmer:luck} in the negative for all four Lehmer's constants over a large class of torsion-free groups (which includes the fundamental group  of the Weeks manifold).

Our first step, of independent interest, consists in computing new values for Fuglede--Kadison determinants of specific operators over free groups:

\begin{theo}[Theorem \ref{thm:detFd}]\label{thm:intro:det}
	Let $d \geqslant 3$. Let $x_1, \ldots, x_{d-1}$ be  $d-1$ generators of the free group $\F_{d-1}$. 
	Let $\zeta_1, \ldots,\zeta_{d-1} \in  \C$ such that $|\zeta_1|=\ldots=|\zeta_{d-1}|=1$.
	Then:
	$$
	\det{}_{\F_{d-1}}(\Id + \zeta_1 R_{x_1} + \ldots + \zeta_{d-1} R_{x_{d-1}}) = \dfrac{(d-1)^{\frac{d-1}{2}}}{d^{\frac{d-2}{2}}}.
	$$
\end{theo}

Our main tools to prove Theorem \ref{thm:intro:det} are relations between Fuglede-Kadison determinants  and combinatorics on Cayley graphs of free groups (specifically works of Bartholdi and Dasbach-Lalin  \cite{Ba, DL}).
Connections between Fuglede-Kadison determinants and random walks on Cayley graphs were surveyed and studied in \cite{KW}.

As a consequence of Theorem \ref{thm:intro:det}, we find a new upper bound $\dfrac{2}{\sqrt{3}}$ for Lehmer's constants over a large class of torsion-free groups:

\begin{theo}[Corollary \ref{cor:lehmer}]\label{cor:intro:free}
	Any torsion-free group $G$ that contains a subgroup $\F_2$ (such as the fundamental group of a hyperbolic $3$-manifold)  satisfies
	$$\Lambda(G), \Lambda_1(G), \Lambda^w(G), \Lambda^w_1(G) \in \left [1, \dfrac{2}{\sqrt{3}}\right ].$$
\end{theo}	

Since $\dfrac{2}{\sqrt{3}}=1.1547... < 1.176... = \mathcal{M}(L)$, Theorem \ref{cor:intro:free} answers Question \ref{qu:intro:lehmer:luck} in the negative for the large class of torsion-free groups with non-cyclic free subgroups. In particular, this class includes the
fundamental group $G_{We}$ of the Weeks manifold, for which all four Lehmer's constants are consequently smaller than $\mathcal{M}(L)$.

In Section \ref{sec:hyp} we go further and find upper bounds smaller than $\frac{2}{\sqrt{3}}$ for an infinite subclass of the class of torsion-free groups with non-cyclic free subgroups. This sub-class includes the fundamental groups of hyperbolic 3-manifolds with volume smaller than $6\pi\ln\left (\frac{2}{\sqrt{3}}\right )$, and notably once again the group $G_{We}$.

\begin{theo}[Theorem \ref{thm:dehn:filling} and Corollary \ref{cor:dehn:filling}]\label{thm:intro:hyp}
	There exists an infinite family of hyperbolic 3-manifolds with finite volume such that for each such manifold $M$, and each group $G$ admitting $\pi_1(M)$ as a subgroup,
	$$ \Lambda_1^w(G) \leqslant \Lambda_1^w(\pi_1(M)) \leqslant \exp \left (\frac{\mathrm{vol}(M)}{6\pi}\right ) < \frac{2}{\sqrt{3}}.$$
\end{theo}

To prove Theorem \ref{thm:intro:hyp}, we prove that the  $L^2$-torsions of the considered manifolds can be expressed as a single Fuglede--Kadison determinant of an operator of size $1$, via the specific example of the Whitehead link and the formula for Dehn surgery for $L^2$-torsions (see Theorem \ref{thm:BA:surgery}).

Finally, in Section \ref{sec:approx}, we present a combinatorial method for computing upper bounds for Fuglede--Kadison determinants over groups with solvable word problem, and we illustrate it for the group of the figure-eight knot.

The article is organised as follows: in Section \ref{sec:prelim} we recall preliminaries on Fuglede--Kadison determinants and Lehmer's constants; 
in Section \ref{sec:detFK} we compute new values of Fuglede--Kadison determinants over free groups via combinatorial group theory; 
in Section \ref{sec:hyp} we review Lehmer's constants for groups of hyperbolic 3-manifolds;
finally, in Section  \ref{sec:approx}, we study upper approximations of Fuglede--Kadison determinants.

In the proof of Theorem \ref{thm:detFd}, several antiderivatives were found using
\textit{Wolfram Alpha} \cite{W}, on a computer equipped with a 
Intel® Core™ i5-8500 CPU @ 3.00GHz × 6 processor.

\section{Preliminaries}\label{sec:prelim}

In this section, we will set some notation and recall some fundamental properties.
We will mostly follow the conventions of \cite{BAC} and \cite{Lu}.

\subsection{Mahler measure of a polynomial}

Let $P\in \C[X_1,\ldots,X_d]$ denote a $d$-variable polynomial. Then its \textit{Mahler measure} $\mathcal{M}(P)$ is the nonnegative real number
$$
\mathcal{M}(P) := \exp\left (\dfrac{1}{(2 \pi)^d} \int_0^{2\pi} \ldots \int_0^{2\pi} \ln\left (\left |P(e^{i \theta_1}, \ldots, e^{i \theta_d})\right |\right )d\theta_1 \ldots d\theta_d \right ) \in \R_{\geqslant 0}.$$
Remark that this definition immediately extends to $d$-variable \textit{Laurent} polynomials $P\in \C\left [X_1^{\pm 1},\ldots,X_d^{\pm 1}\right ]$. This will be useful in Section \ref{sub:fkdet}, where the group algebra $\C[\Z^d]$ will be naturally identified with the algebra of Laurent polynomials $\C\left [X_1^{\pm 1},\ldots,X_d^{\pm 1}\right ]$. We refer to \cite{Bo} (among others) for a survey on Mahler measures.

\begin{exem}
	For $d=1$, and $P(X) = C \cdot X^{-l} \cdot \prod_{j=1}^r (X - \alpha_j) \in \C[X^{\pm 1}]$ (where $C, \alpha_1, \ldots,\alpha_r \in \C$ and $l \in \N$), there is a closed formula
	$$\mathcal{M}(P)= |C| \cdot \prod_{j=1}^r \max\{1,|\alpha_j|\}.$$
\end{exem}

When $P$ has two or more variables, the Mahler measure is known for several classes of examples, such as the following one:

\begin{exem}[\cite{Bo}, Section 4]\label{ex:mahler:boyd}
	The two-variable polynomial $1+X+Y \in \C[X,Y]$ has Mahler measure 
	$\mathcal{M}(1+X+Y) = e^{
		\frac{1}{\pi}\Im \mathrm{Li_2}\left (e^{i \pi/3}\right )
	}
	= 1.38135...$
\end{exem}

\begin{exem}
	\textit{Lehmer's polynomial} $L(X) \in \Z[X]$ is defined as
	$$L(X) := X^{10}+ X^9- X^7- X^6- X^5- X^4- X^3+ X+1.$$
	Its Mahler measure is equal to
	$\mathcal{M}(L)=1.176280818...$
\end{exem}

\begin{question}[Lehmer's problem]\label{qu:lehmer} \ 
	
	(1) 	Does there exist a constant $\Lambda >1$ such that for all Laurent polynomials with integer coefficients $P(X) \in \Z[X, X^{-1}]$ with $\mathcal{M}(P)>1$, we have
	$$\mathcal{M}(P)>\Lambda?$$
	
	(2) If the answer to (1) is yes, is the constant $\Lambda$ equal to the Mahler measure of Lehmer's polynomial, i.e.
	$$\Lambda = \mathcal{M}(L)=1.176280818... ?$$
\end{question}

\subsection{Fuglede-Kadison determinant}\label{sub:fkdet}

In this section we will give  short definitions of the von Neumann trace and the Fuglede-Kadison determinant.
More details can be found in \cite{Lu} and \cite{BAC}.

Let $G$ be a finitely generated group.
The Hilbert space $\ell^2(G)$ is the completion of the group algebra $\C G$, and the space of bounded operators on it is denoted $B(\ell^2(G))$. We will focus on \textit{right-multiplication operators} $R_w \in B(\ell^2(G))$, where $R_{\cdot}$ denotes the right regular action of $G$ on $\ell^2(G)$ extended to the group ring $\C G$ (and further extended to the rings of matrices $M_{p,q}(\C G)$).

For any element $w = a_0 \cdot 1_G + a_{1} g_1 \ldots + a_{r} g_r \in \C G$, the \textit{von Neumann trace} $\mathrm{tr}_G$ of the associated right multiplication operator is defined as
$$\mathrm{tr}_G(R_w) = \mathrm{tr}_G\left (
a_0 \Id_{\ell^2(G)} + a_{1} R_{g_1} \ldots + a_{r} R_{g_r} \right ) := a_0,$$
and the von Neumann trace for a finite square matrix over $\C G$ is given as the sum of the traces of the diagonal coefficients.

Now the most concise definition of the \textit{Fuglede-Kadison determinant} $\det_G(A)$ of a  right-multiplication operator $A$ is probably 
$$
\det{}_G(A) := 
\lim_{\varepsilon \to 0^+}
\left (\exp \circ \left (\dfrac{1}{2} \mathrm{tr}_G\right ) \circ  \ln \right )
\left (
(A_\perp)^*(A_\perp) + \varepsilon \Id \right )  \ \geqslant 0,$$
where $A_\perp$ is the restriction of $A$ to an orthogonal complement of the kernel, $*$ is the adjunction and $\ln$ the logarithm of an operator in the sense of the holomorphic functional calculus. Compare with \cite[Theorem 3.14]{Lu} and Proposition \ref{prop:detFK:properties} below. We call the operator $A$ \textit{of determinant class} if $\det_G(A) \neq 0$.

The following properties concern the classical Fuglede-Kadison determinant $\det_G$ described in \cite[Chapter 3]{Lu}, which is not always multiplicative  if one deals with non injective operators. Moreover, this determinant forgets about the influence of the spectral value $0$, which surprisingly makes it take the value $1$ for the zero operator. More recent articles have used the \textit{regular Fuglede-Kadison determinant} $\det^r_G$ instead, which is defined for square injective operators, is zero for non injective operators, and is always multiplicative. In this paper we will work with both types of determinants, but the reader should be reassured that most of the statements we will make remain unchanged while replacing one determinant with the other (up to assumptions on injectivity usually). Similarly, the statements of the following Proposition \ref{prop:detFK:properties} admit immediate variants with $\det^r_G$. All statements of Proposition \ref{prop:detFK:properties} follow from \cite[Chapter 3]{Lu}, except for (6), which directly follows from the others.

\begin{prop}[\cite{Lu}]\label{prop:detFK:properties}
	Let~$G$ be a countable discrete group and let $$A,B,C,D \in \sqcup_{p,q \in \N} R_{M_{p,q}(\C G)}$$ be general right multiplication operators. The Fuglede-Kadison determinant satisfies the following properties:
	\begin{enumerate}
		\item (multiplicativity) If $A, B$ are injective, square  and of the same size, then 
		$$\det{}_G(A \circ B) = \det{}_G(A) \det{}_G(B).$$
		\item (block triangular case) If $A, B$ are injective and square, then
		$$\det{}_G
		\begin{pmatrix}
		A & C \\ 0& B
		\end{pmatrix}
		= \det{}_G(A) \det{}_G(B),$$
		where $C$ has the appropriate dimensions.
		\item  (induction) If $\iota\colon G \hookrightarrow H$ is a group monomorphism, then 
		$$\det{}_H (\iota (A)) 
		= \det{}_G(A).$$
		\item  (relation with the von Neumann trace) If $A$ is a positive operator, then
		$$\det{}_G(A) = \left (\exp \circ \mathrm{tr}{}_G \circ \ln\right )(A).$$
		\item (simple case) If~$g \in G$ is of infinite order, then for all~$t \in \C$ the operator $ \Id - t R_g$ is injective and~$$
		\mathrm{det}_{G} ( \Id - t R_g) = \max ( 1 , |t|).$$
		\item ($2\times 2$ trick) For all $A,B,C,D \in N(G)$ such
		that $B$ is invertible, 
		$\begin{pmatrix}
		A & B \\ C& D
		\end{pmatrix}$ is injective if and only if 
		$D B^{-1} A - C$ is injective, and in this case
		one has:
		$$\det{}_G
		\begin{pmatrix}
		A & B \\ C& D
		\end{pmatrix}
		= \det{}_G(B) \det{}_G(D B^{-1} A - C).$$
		\item (relation with Mahler measure) Let $G= \Z^d$, and $P\in \C[X_1^{\pm 1},\ldots,X_d^{\pm 1}]$ denote the  Laurent polynomial associated to the operator $A \in R_{\C \Z^d}$. Then
		$$
		\mathrm{det}_{\Z^d} (A) = \mathcal{M}(P) = \exp\left (\dfrac{1}{(2 \pi)^d} \int_0^{2\pi} \ldots \int_0^{2\pi} \ln\left (\left |P(e^{i \theta_1}, \ldots, e^{i \theta_d})\right |\right )d\theta_1 \ldots d\theta_d \right ),$$
		where $\mathcal{M}$ is the Mahler measure.
		\item (limit of positive operators) If $A$ is injective, then
		$$\det{}_{G}(A) = \lim_{\varepsilon \to 0^+}
		\sqrt{\det{}_{G}(A^* A + \varepsilon \Id)}.$$
		\item (dilations) Let $\lambda \in \C^*$. Then:
		$$\det{}_G\left (\lambda \ \Id^{\oplus n}\right ) = |\lambda|^n.$$
	\end{enumerate}
\end{prop}

\begin{rema}\label{rem:atiyah}
	If the group $G$ satisfies the \textit{strong Atiyah conjecture} (see \cite[Chapter 10]{Lu}), then the right multiplication operator by any non-zero element of $\C G$ is injective, which makes it convenient to apply some parts of Proposition \ref{prop:detFK:properties}. Note that free groups and free abelian groups satisfy the strong Atiyah conjecture.
\end{rema}

\subsection{Lehmer's constants}\label{sub:lehmer:luck}

Given a group $G$, its  \textit{Lehmer's constants}, as defined by Lück in \cite{Lu2}, are:
\begin{itemize}
	\item $\Lambda(G):=\inf\left \{ \det_G(A) \ | \ A \in \sqcup_{p,q \in \N} R_{M_{p,q}(\Z G)}, \ \det_G(A)>1 \right \}$,
	\item $\Lambda^w(G):=\inf\left \{ \det_G(A) \ | \ A \in \sqcup_{n \in \N} R_{M_{n}(\Z G)}, \ A \text{ injective}, \ \det_G(A)>1 \right \}$,
	\item $\Lambda_1(G):=\inf\left \{ \det_G(A) \ | \ A \in R_{\Z G}, \ \det_G(A)>1 \right \}$,
	\item $\Lambda^w_1(G):=\inf\left \{ \det_G(A) \ | \ A \in R_{\Z G}, \ A \text{ injective}, \ \det_G(A)>1 \right \}$.
\end{itemize}
Observe that 
$1 \leqslant \Lambda(G) \leqslant \Lambda^w(G) \leqslant \Lambda^w_1(G)$
and $1 \leqslant \Lambda(G) \leqslant \Lambda_1(G) \leqslant \Lambda^w_1(G)$.

\begin{rema}\label{rem:lehmer:subgroup}
	If $H$ is a subgroup of $G$, it follows from Proposition \ref{prop:detFK:properties} (3) that $\lambda(G) \leqslant \lambda(H)$ for any $\lambda \in \{\Lambda, \Lambda_1, \Lambda^w, \Lambda^w_1\}$.
\end{rema}

\begin{rema}
	Lehmer's constants are motivated by  generalisations of the well-known Lehmer problem (see Question \ref{qu:lehmer}).
	Indeed, Question \ref{qu:lehmer} can be rephrased as
	\begin{enumerate}
		\item Do we have $\Lambda_1^w(\Z)>1$?
		\item Do we have $\Lambda_1^w(\Z)=\mathcal{M}(L)=1.176280818...$?
	\end{enumerate}
\end{rema}

As explained in \cite{Lu2}, if there is no upper bound on the order of finite subgroups of $G$ (for instance if $G$ is the lamplighter group), then all four Lehmer's constants of $G$ are equal to $1$. Hence
Lehmer's constants are especially interesting to study for \textit{torsion-free} groups. For this class of groups, Lück stated the following natural questions:
\begin{question}[\cite{Lu2}]\label{qu:lehmer:luck} \
	\begin{enumerate}
		\item For which torsion-free groups $G$ are
		all four Lehmer's constants $\Lambda(G)$, $\Lambda^w(G)$, $\Lambda_1(G)$ and $\Lambda^w_1(G)$  greater than $1$?
		\item For which torsion-free groups $G$  do we have
		$$  \Lambda(G) = \Lambda^w(G) = \Lambda_1(G)= \Lambda^w_1(G) = \mathcal{M}(L)=1.17628... \ ?$$
	\end{enumerate}
\end{question}

\begin{exem}[\cite{Lu2}, Example 13.2]\label{ex:weeks}
	Lück provided a partial negative answer to Question \ref{qu:lehmer:luck} (2) by proving that $$\Lambda(G) \leqslant \Lambda^w(G) \leqslant
	\exp\left (\frac{\mathrm{vol}(M_{We})}{6 \pi}\right ) =
	1.05128... < \mathcal{M}(L)=1.17628...$$ for $G=\pi_1(M_{We})$ the fundamental group of the hyperbolic Weeks manifold $M_{We}$. This manifold is the closed hyperbolic 3-manifold with smallest hyperbolic volume $\mathrm{vol}(M_{We})= 0.942707...$ among this class, and it can be obtained by doing Dehn fillings on both boundary components of the Whitehead link exterior.
	
	It follows from Remark \ref{rem:lehmer:subgroup} that the same upper bound on $\Lambda^w(.)$ is true for any group $G'$ containing $G$ as a subgroup.
	
	In Section \ref{sec:hyp} we will extend this result by proving that this upper bound applies to $\Lambda_1^w(G)$, and thus to all four Lehmer's constants of $G$.
\end{exem}

\subsection{$L^2$-torsions}

At first read, the reader can skip this section and simply think of the \textit{$L^2$-torsion} $T^{(2)}(M)$ of a hyperbolic $3$-manifold $M$ as a positive number computed from Fuglede--Kadison determinants over $\pi_1(M)$ and related to the hyperbolic volume of $M$.
We state here the necessary notions for Section \ref{sec:hyp} and we refer to \cite{Lu} and \cite{BA2} for more details.

A \textit{finitely generated Hilbert $\mathcal{N}(G)$-module} is a hilbert space $V$ on which there is a left $G$-action by isometries, and 
such that there exists a positive integer $m$ and an embedding $\phi$ of $V$ into $\bigoplus_{i=1}^m \ell^2(G)$ (in this paper, such spaces $V$ will always be of the form $\ell^2(G)^{\oplus n}$ for $n\in \N$). 

For $U$ and $V$ two finitely generated Hilbert $\mathcal{N}(G)$-modules, we will call $ f\colon U \rightarrow V$ a \textit{morphism of finitely generated Hilbert $\mathcal{N}(G)$-modules} if $f$ is a linear $G$-equivariant map, bounded for the respective scalar products of $U$ and $V$ (in this paper, these morphisms will simply be right multiplication operators).

A \textit{finite Hilbert $\NN(G)$-chain complex} $C_*$ is a sequence of morphisms of finitely generated Hilbert $\NN(G)$-modules
$$C_* = 0 \to C_n \overset{\partial_n}{\longrightarrow} C_{n-1} 
\overset{\partial_{n-1}}{\longrightarrow} \ldots
\overset{\partial_2}{\longrightarrow} C_1 \overset{\partial_1}{\longrightarrow} C_0 \to 0$$
such that $\partial_p \circ \partial_{p+1} =0$ for all $p$ (in this paper, $n$ will be at most $3$).

The \textit{$p$-th $L^2$-homology of $C_*$}  
$H_p^{(2)}(C_*) := \ker(\partial_p) / \overline{\mathrm{Im}(\partial_{p+1})}$ is a finitely generated Hilbert $\NN(G)$-module. 
We say that $C_*$ is \textit{weakly acyclic} if its $L^2$-homology is trivial. We say that $C_*$ is of \textit{determinant class} if all the operators $\partial_p$ are of determinant class.

Let $C_*$ be a finite Hilbert $\NN(G)$-chain complex as above. Its \textit{$L^2$-torsion} is
$$T^{(2)}(C_*) := \prod_{i=1}^n \det {}_{\NN(G)}(\partial_i)^{(-1)^i} \in \R_{>0}$$
if $C_*$ is weakly acyclic and of determinant class, and is  $T^{(2)}(C_*):=0$ otherwise.

Let $X$ be a connected compact CW-complex, $\pi=\pi_1(X)$ and  $\gamma\colon \pi \to G$ be a group homomorphism.
The cellular chain complex of $\widetilde{X}$ denoted
$$C_*(\widetilde{X},\Z) = 
\left (\ldots \to \bigoplus_i \Z[\pi] \widetilde{e}_i^k \to \ldots\right )$$
is a chain complex of left $\Z[\pi]$-modules.
Here the $\widetilde{e}_i^k$ are lifts of the cells $e_i^k$ of $X$.
The group $\pi$ acts on the right on $\ell^2(G)$ by $g \mapsto R_{\gamma(g)}$, an action which induces a structure of right $\Z[\pi]$-module on $\ell^2(G)$.
Let
$$C_*^{(2)}(X,\gamma) = \ell^2(G) \otimes_{\Z[\pi]}  C_*(\widetilde{X},\Z)$$
denote the finite Hilbert $\NN(G)$-chain complex obtained by tensor product via these left- and right-actions; we call $C_*^{(2)}(X,\gamma)$ \textit{an $\NN(G)$-cellular chain complex of $X$}. 
We then denote
$$ T^{(2)}(X,\gamma) := T^{(2)}\left (C_*^{(2)}(X,\gamma)\right )$$
the \textit{twisted $L^2$ torsion of $X$ by $\gamma$}. It is non-zero if and only if $C_*^{(2)}(X,\gamma)$ is weakly acyclic and of determinant class.

From classical theorems due to Chapman and Cohen, any two CW-structures $X, Y$ on a compact smooth 3-manifold $M$ are simple homotopy equivalent, which implies that their twisted $L^2$-torsions satisfy
$T^{(2)}(Y,\gamma) = T^{(2)}(X, \gamma \circ f_\sharp),$
where $f: X \to Y$ is a simple homotopy equivalence, $f_\sharp$ the induced isomorphism of fundamental groups, and $\gamma: \pi_1(Y) \to G$ a group homomorphism (see for example \cite[Section 2.1.3]{BA} or \cite[Chapter 3]{Lu} for further details and references). As a consequence, in the remainder of the paper, for any \textit{hyperbolic 3-manifold} $M$ (i.e. a compact connected oriented 3-manifold with empty or toroidal boundary, whose \textit{interior} admits a complete hyperbolic structure of finite volume), we will write $T^{(2)}(M,\gamma)$ for simplicity, without making the CW-structure or the group isomorphism explicit.

\begin{theo}[Lück--Schick \cite{LS}]\label{thm:LS}
	Let $M$ be a hyperbolic $3$-manifold of finite volume $\mathrm{vol}(M)$. Then
	$$T^{(2)}(M)=\exp\left (\dfrac{\mathrm{vol}(M)}{6\pi}\right ).$$
\end{theo}

\begin{theo}[Surgery formula \cite{BA2}]\label{thm:BA:surgery}
	Let $M, N$ be hyperbolic $3$-manifolds of finite volume such that $M$ is obtained from $N$ by Dehn filling on a toroidal boundary component.
	Let $Q\colon \pi_1(N) \twoheadrightarrow \pi_1(M)$ denote the group epimorphism induced by the Dehn filling.
	Then
	$$T^{(2)}(N,Q)=T^{(2)}(M)=
	\exp\left (\dfrac{\mathrm{vol}(M)}{6\pi}\right ).$$
\end{theo}

\begin{rema}
	Initially, a general twisted $L^2$-torsion $T^{(2)}(M,\gamma)$ is defined as an alternated product of several Fuglede--Kadison determinants of operators of various sizes. However, in several important cases (see Section \ref{sec:hyp}), convenient cellular decompositions and determinant simplifications help us see that $T^{(2)}(M,\gamma)$ can be expressed as a single Fuglede--Kadison determinant of a square operator of size $1$.
\end{rema}

\section{Fuglede-Kadison determinants over free groups}\label{sec:detFK}

In this section, we present a general method to compute Fuglede-Kadison determinants via counting paths on Cayley graphs, and we apply this method to symmetric operators over free groups (studied by Bartholdi and Dasbach-Lalin).

\subsection{Computing Fuglede-Kadison determinants from Cayley graphs}

The following proposition gives a method to compute the Fuglede-Kadison determinant $\det_G(A)$ using the generating series $u_{A^* A}(t)$ associated to the number of closed paths on a Cayley graph associated to $G$ and $A^* A$. This method was studied by Dasbach-Lalin in \cite{DL} and by Lück in \cite[Section 3.7]{Lu} with slight differences in notation, and we provide a complete proof here for the reader's convenience.

\begin{prop}\label{prop:det_traces}
	Let $G$ be a finitely presented group and let $A \in R_{\C G}$ denote an injective right multiplication operator on $\ell^2(G)$ by a non-zero element of the group algebra. 
	Then for any $\lambda \in \left (0,\|A\|^{-2}\right )$, we have:
	\begin{align*}
	\det{}_{G}(A) &= \lim_{\varepsilon \to 0^+} \dfrac{1}{\sqrt{\lambda}} \exp \left ( -\frac{1}{2}
	\int_0^1 \dfrac{w_{\lambda,\varepsilon}(t)-1}{t} dt
	\right )\\
	&=\lim_{\varepsilon \to 0^+} \dfrac{1}{\sqrt{\lambda}} \exp \left ( -\frac{1}{2}
	\int_0^1 \dfrac{\frac{1}{1-(1-\lambda \varepsilon)t} \ u_{A^*A}\left (\dfrac{-\lambda t}{1-(1-\lambda \varepsilon)t}\right )-1}{t} dt
	\right ),
	\end{align*}
	where
	$$
	u_{A^* A}(t):= \sum_{k=0}^{\infty} 
	\mathrm{tr}_{G}\left ( (A^* A)^k
	\right ) t^k
	$$
	is a convergent power series for $t$ small enough,
	and
	$$w_{\lambda,\varepsilon}(t) = \sum_{n=0}^{\infty}
	\mathrm{tr}_{G}\left (
	\left (
	(1-\lambda \varepsilon) \Id - \lambda A^* A
	\right )^n
	\right ) t^n=\frac{1}{1-(1-\lambda \varepsilon)t} \ u_{A^*A}\left (\dfrac{-\lambda t}{1-(1-\lambda \varepsilon)t}\right )$$
	is a convergent power series for $\varepsilon$ small enough and $|t|<1$ .
\end{prop}

\begin{proof}
	\underline{First equality:}
	Let $\lambda \in \left (0,\|A\|^{-2}\right )$.
	Since $A$ is  injective,  Proposition \ref{prop:detFK:properties} (8) implies that
	$$
	\det{}_{G}(A) = \lim_{\varepsilon \to 0^+}
	\sqrt{\det{}_{G}(A^* A + \varepsilon \Id)}.$$
	Let $\varepsilon >0$ such that $\lambda < \dfrac{1}{\|A\|^{2}+\varepsilon}<\dfrac{1}{\|A\|^{2}}$.
	Since $A^* A + \varepsilon \Id$ is  positive, we obtain:
	\begin{align*}
	\det{}_{G}(A^* A + \varepsilon \Id) 
	&= \left (\exp \circ \mathrm{tr}_{G} \circ \ln\right ) (A^* A + \varepsilon \Id) \\
	&= \dfrac{1}{\lambda}\left (\exp \circ \mathrm{tr}_{G} \circ \ln\right ) (\lambda A^* A + \lambda \varepsilon \Id)\\
	&= \dfrac{1}{\lambda}\left (\exp \circ \mathrm{tr}_{G} \circ \ln\right ) (\Id -(\Id -(\lambda A^* A + \lambda \varepsilon \Id)))\\
	&= \dfrac{1}{\lambda}\left (\exp \circ \mathrm{tr}_{G} \right ) \left (-\sum_{n=1}^{\infty} \frac{1}{n}(\Id -(\lambda A^* A + \lambda \varepsilon \Id))^n\right )\\
	&= \dfrac{1}{\lambda}\exp  \left (-\sum_{n=1}^{\infty} \frac{1}{n}\mathrm{tr}_{G}\left (
	\left (
	(1-\lambda \varepsilon) \Id - \lambda A^* A
	\right )^n
	\right )\right ),
	\end{align*}
	where the first equality follows from Proposition \ref{prop:detFK:properties} (4), the fourth one from holomorphic functional calculus and the fact that the spectrum of the positive operator $\lambda A^* A + \lambda \varepsilon \Id$ is inside $(0,1)$ (since
	$\lambda < \dfrac{1}{\|A\|^{2}+\varepsilon}$), and the fifth one from the continuity property of the von Neumann trace.
	
	Now, since the series $\sum_{n=1}^{\infty} \frac{1}{n}\mathrm{tr}_{G}\left (
	\left (
	(1-\lambda \varepsilon) \Id - \lambda A^* A
	\right )^n
	\right )$ converges, then
	$$w_{\lambda,\varepsilon}(t) := \sum_{n=0}^{\infty}
	\mathrm{tr}_{G}\left (
	\left (
	(1-\lambda \varepsilon) \Id - \lambda A^* A
	\right )^n
	\right ) t^n$$
	is a convergent power series for $|t|<1$, and moreover that for all $T \in (0,1)$,
	$$\sum_{n=1}^{\infty} \frac{1}{n}\mathrm{tr}_{G}\left (
	\left (
	(1-\lambda \varepsilon) \Id - \lambda A^* A
	\right )^n
	\right ) T^n = \int_0^T \dfrac{w_{\lambda,\varepsilon}(t)-1}{t} dt.$$
	Finally, once again since 	$\sum_{n=1}^{\infty} \frac{1}{n}\mathrm{tr}_{G}\left (
	\left (
	(1-\lambda \varepsilon) \Id - \lambda A^* A
	\right )^n
	\right )$ converges, we can apply Abel's theorem and make the passage $T\to 1^-$ in the previous equality. Hence:
	\begin{align*}
	\det{}_{G}(A) &= \lim_{\varepsilon \to 0^+}
	\sqrt{\det{}_{G}(A^* A + \varepsilon \Id)}\\
	&= \lim_{\varepsilon \to 0^+}
	\sqrt{\dfrac{1}{\lambda}\exp  \left (-\sum_{n=1}^{\infty} \frac{1}{n}\mathrm{tr}_{G}\left (
		\left (
		(1-\lambda \varepsilon) \Id - \lambda A^* A
		\right )^n
		\right )\right )}\\
	&= \lim_{\varepsilon \to 0^+}
	\sqrt{\dfrac{1}{\lambda}\exp  \left (-\int_0^1 \dfrac{w_{\lambda,\varepsilon}(t)-1}{t} dt\right )},
	\end{align*}
	and the first equality follows.
	
	\underline{Second equality:} 	Now we compute, for $\varepsilon>0$ small enough and $t \in [0,1)$:
	\begin{align*}
	w_{\lambda,\varepsilon}(t) &= \sum_{n=0}^{\infty}
	\mathrm{tr}_{G}\left (
	\left (
	(1-\lambda \varepsilon) \Id - \lambda A^* A
	\right )^n
	\right ) t^n\\
	&= \sum_{n=0}^{\infty}
	\mathrm{tr}_{G}\left (
	\sum_{k=0}^n \binom{n}{k}
	(1-\lambda \varepsilon)^{n-k} (-\lambda)^k (A^* A)^k
	\right ) t^n\\
	&= \sum_{n=0}^{\infty} \sum_{k=0}^n
	\binom{n}{k}
	(1-\lambda \varepsilon)^{n-k} (-\lambda)^k \mathrm{tr}_{G}\left ( (A^* A)^k
	\right ) t^n\\
	&= \sum_{k=0}^{\infty} 
	\left ( \dfrac{-\lambda}{1-\lambda \varepsilon} \right )^k \mathrm{tr}_{G}\left ( (A^* A)^k
	\right )
	\sum_{n=k}^\infty
	\binom{n}{k}
	(1-\lambda \varepsilon)^{n}  t^n\\
	&= \sum_{k=0}^{\infty} 
	\left ( \dfrac{-\lambda}{1-\lambda \varepsilon} \right )^k \mathrm{tr}_{G}\left ( (A^* A)^k
	\right )
	\dfrac{\left ((1-\lambda \varepsilon)t\right )^k}{\left (1-(1-\lambda \varepsilon)t\right )^{k+1}}
	\\
	&= \dfrac{1}{1-(1-\lambda \varepsilon)t} \sum_{k=0}^{\infty} 
	\left ( \dfrac{-\lambda t}{1-(1-\lambda \varepsilon)t} \right )^k \mathrm{tr}_{G}\left ( (A^* A)^k
	\right ),
	\end{align*}
	where the fifth equality follows from  the 
	binomial formula $$
	\dfrac{1}{(1-X)^\alpha}=\sum_{j=0}^\infty \binom{\alpha-1+j}{j} X^j,$$
	which itself follows from the usual Taylor series for $(1-X)^{-\alpha}$ and the binomial coefficient identity $ \binom{-\alpha}{j} = \binom{\alpha-1+j}{j} (-1)^j$.

	Thus, by denoting $u(t):= \sum_{k=0}^{\infty} 
	\mathrm{tr}_{G}\left ( (A^* A)^k
	\right ) t^k,$ we have
	$$w_{\lambda,\varepsilon}(t) = 
	\frac{1}{1-(1-\lambda \varepsilon)t} \ u\left (\dfrac{-\lambda t}{1-(1-\lambda \varepsilon)t}\right )$$
	and the second equality follows.
\end{proof}

\subsection{Fuglede--Kadison determinants for non-cyclic free groups}

The following proposition follows from works of Bartholdi and Dasbach-Lalin \cite{Ba, DL} on counting paths on regular trees. 

\begin{prop}\label{prop:series:free}
	Let $d \geqslant 3$,  let $x_1, \ldots, x_{d-1}$ be  $d-1$ generators of the free group $\F_{d-1}$, and let $\zeta_1, \ldots,\zeta_{d-1} \in  \C$ such that $|\zeta_1|=\ldots=|\zeta_{d-1}|=1$.
	
	For $G=\F_{d-1}$, $A=\Id + \zeta_1 R_{x_1} + \ldots + \zeta_{d-1} R_{x_{d-1}}$, and $t$ small enough, the following generating series is equal to:
	$$u_{A^* A}(t) = \sum_{k=0}^{\infty} 
	\mathrm{tr}_{G}\left ( (A^* A)^k
	\right ) t^k  = \dfrac{2d-2}{d-2+d\sqrt{1-4(d-1)t}}.$$ 
\end{prop}

\begin{proof}
	The case where $\zeta_1=\ldots=\zeta_{d-1}=1$ is proven in \cite{Ba, DL}. The main idea is the fact that $\mathrm{tr}_{G}\left ( (A^* A)^k
	\right )$  counts the number of closed paths on a Cayley graph associated to $G$ with generating set related to the operator $A$. In general, one can pick from several different Cayley graphs for this purpose, see \cite{Ba, DL} for specific examples and \cite{KW} for a more canonical choice of Cayley graph. As explained in \cite{Ba, DL}, we can pick a $d$-regular tree for the present problem. The generating series $u_{A^* A}(t)$ is then determined as the solution of a  functional equation, which is found by listing every possible form of a closed path on the graph.
	
	Let us consider the general case.
	For $k \in \N$, the coefficient $\mathrm{tr}_{G}\left ( (A^* A)^k
	\right )$ is equal to the sum of all terms $\zeta_{j_1}\overline{\zeta_{j_2}}\ldots \zeta_{j_{2k-1}}\overline{\zeta_{j_{2k}}}$ (where $j_{1}, \ldots, j_{2k} \in \{0,1, \ldots,d-1\}$ and $\zeta_0 :=1$) such that
	$x_{j_1}x_{j_2}^{-1}\ldots x_{j_{2k-1}}x_{j_{2k}}^{-1} \in \F_{d-1}$ is trivial (with the convention $x_0:=1$). Now, if $x_{j_1}x_{j_2}^{-1}\ldots x_{j_{2k-1}}x_{j_{2k}}^{-1}$ is trivial, then each non-zero index $j_i$ is paired  with another index $j_{i'}$ with $i' \neq i, j_{i'}=j_i$. Hence the associated term $\zeta_{j_1}\overline{\zeta_{j_2}}\ldots \zeta_{j_{2k-1}}\overline{\zeta_{j_{2k}}}$ is equal to $1$, like in the first case, and the result follows similarly.
\end{proof}

In the following theorem, we compute Fuglede-Kadison determinants of basic operators on the free groups, using the explicit generating series of Proposition \ref{prop:series:free}.

\begin{theo}\label{thm:detFd}
	Let $d \geqslant 3$, and let $x_1, \ldots, x_{d-1}$ be  $d-1$ generators of the free group $\F_{d-1}$.
	Let $\zeta_1, \ldots,\zeta_{d-1} \in  \C$ such that $|\zeta_1|=\ldots=|\zeta_{d-1}|=1$.
	Then we have:
	$$
	\det{}_{\F_{d-1}}(\Id + \zeta_1 R_{x_1} + \ldots + \zeta_{d-1} R_{x_{d-1}}) = \dfrac{(d-1)^{\frac{d-1}{2}}}{d^{\frac{d-2}{2}}}.
	$$
	
	In particular, for any two generators $x,y$ of the free group $\F_2$, we have:
	$$
	\det{}_{\F_2}(\Id + R_x + R_y) = \dfrac{2}{\sqrt{3}} = 1.1547... 
	$$
\end{theo}

\begin{proof}
	Let $d \geqslant 3$ and $\zeta_1, \ldots,\zeta_{d-1}$ in the  unit circle. First we use Proposition \ref{prop:det_traces}, by taking $G=\F_{d-1}$, $A=\Id + \zeta_1 R_{x_1} + \ldots + \zeta_{d-1} R_{x_{d-1}}$, $\lambda \in (0,\frac{1}{d^2})$, and we obtain:
	\begin{align*}
	\det{}_{\F_{d-1}}(\Id + \zeta_1 R_{x_1} + \ldots + \zeta_{d-1} R_{x_{d-1}}) &= \det{}_{G}(A)\\
	&= \lim_{\varepsilon \to 0^+} \dfrac{1}{\sqrt{\lambda}} \exp \left ( -\frac{1}{2}
	\int_0^1 \dfrac{w_{\lambda,\varepsilon}(t)-1}{t} dt
	\right ),
	\end{align*}
	where  $$w_{\lambda,\varepsilon}(t) = \sum_{n=0}^{\infty}
	\mathrm{tr}_{G}\left (
	\left (
	(1-\lambda \varepsilon) \Id - \lambda A^* A
	\right )^n
	\right ) t^n$$
	for $\varepsilon>0$ small enough and $t \in [0,1)$.
	
	From Proposition \ref{prop:det_traces}, by denoting $u_{A^*A}(t):= \sum_{k=0}^{\infty} 
	\mathrm{tr}_{G}\left ( (A^* A)^k
	\right ) t^k,$ we have
	$$w_{\lambda,\varepsilon}(t) = 
	\frac{1}{1-(1-\lambda \varepsilon)t} \ u_{A^*A}\left (\dfrac{-\lambda t}{1-(1-\lambda \varepsilon)t}\right ).$$
	
	Since it suffices to prove that 
	$$\lim_{\varepsilon \to 0^+} \int_0^1 \dfrac{w_{\lambda,\varepsilon}(t)-1}{t} dt = \ln\left (\frac{d^{d-2}}{(d-1)^{d-1} \lambda}\right ),$$
	let us denote
	$I_{\lambda,\varepsilon}:=\int_0^1 \frac{w_{\lambda,\varepsilon}(t)-1}{t} dt$ and prove that 
	$I_{\lambda,\varepsilon} \underset{\varepsilon \to 0^+}{\to} \ln\left (\frac{d^{d-2}}{(d-1)^{d-1} \lambda}\right )$.
	
	It follows from Proposition \ref{prop:series:free} that
	$$u_{A^*A}(t) = \dfrac{2d-2}{d-2+d\sqrt{1-4(d-1)t}},$$ 
	thus
	$$I_{\lambda,\varepsilon}=\int_0^1 \frac{\frac{1}{1-(1-\lambda \varepsilon)t} \ \frac{2d-2}{d-2+d\sqrt{1+\frac{4(d-1)\lambda t}{1-(1-\lambda \varepsilon)t}}} -1}{t} dt.$$
	With the change of variables $x=(1-\lambda \varepsilon)	t$ and by denoting $a=\frac{4 (d-1) \lambda}{1-\lambda \varepsilon}$, we find
	$I_{\lambda,\varepsilon}= \int_{0}^{1-\lambda \varepsilon}\left (2(d-1)R(x)-\frac{1}{x}\right )dx,$
	where 
	$$R(x):=\dfrac{1}{x(1-x)\left (d-2+d\sqrt{1+\frac{ax}{1-x}}\right )}
	=\dfrac{d\sqrt{1+\frac{ax}{1-x}}-(d-2)}{x\left ((d^2 a-4d+4)x+4d-4\right )}.$$
	To find an antiderivative of $R(x)$, we first split it into a sum of partial fractions:
	$$R(x)=\dfrac{1}{4(d-1)}	\left (
	-\dfrac{d-2}{x} +  \dfrac{d-2}{x+\frac{4d-4}{d^2 a-4d+4}}+d\dfrac{\sqrt{1+\frac{ax}{1-x}}}{x}-d\dfrac{\sqrt{1+\frac{ax}{1-x}}}{x+\frac{4d-4}{d^2 a-4d+4}}
	\right ).$$
	For any generic constant $C$,  an antiderivative of
	$x \mapsto \dfrac{\sqrt{1+\frac{ax}{1-x}}}{x+\frac{4C-4}{C^2 a-4C+4}}$ is given by
	$$
	x \mapsto 2\sqrt{1-a} \cdot  \arsinh \left (\sqrt{\dfrac{(1-x)(1-a)}{a}}\right )
	-2\frac{C-2}{C}\artanh \left (
	\frac{C-2}{C} \sqrt{\dfrac{1-x}{ax-x+1}}
	\right ), 
	$$	
	hence the cases $C=1$ and $C=d$ provide the antiderivatives of the third and fourth terms in the previous expression of $R(x)$. The $\arsinh$ terms cancel, and we find the following antiderivative $F(x)$ for the function $\left (2(d-1)R(x)-\frac{1}{x}\right )$:
	\begin{align*}
	F(x)=&
	-\frac{d}{2}\ln(x)+\frac{d-2}{2}\ln\left |x+\frac{4d-4}{d^2 a-4d+4}\right |\\
	&-d\artanh\left (\sqrt{\dfrac{1-x}{ax-x+1}}\right )
	+(d-2)\artanh\left (\frac{d-2}{d} \sqrt{\dfrac{1-x}{ax-x+1}}\right ).
	\end{align*}
	From what precedes, we therefore have $I_{\lambda,\varepsilon}=F(1-\lambda\varepsilon)-F(0)$.
	
	Recall that $a, R, F$ all depend on $\varepsilon$, although it is not apparent in the notation.
	
	Since $\lim_{\varepsilon \to 0^+}(a) = 4(d-1)\lambda$, we can compute
	$$\lim_{\varepsilon \to 0^+}F(1-\lambda\varepsilon) 
	= \frac{d-2}{2}\ln\left |\frac{d^2 4(d-1)\lambda}{d^2 4(d-1)\lambda-4d+4}\right |.
	$$
	
	To compute $F(0)$, we first need to remark that
	\begin{align*}
	&\artanh\left (\sqrt{\dfrac{1-x}{ax-x+1}}\right ) 
	= \frac{1}{2}\ln\left (\dfrac{1+\sqrt{\dfrac{1-x}{ax-x+1}}}{1-\sqrt{\dfrac{1-x}{ax-x+1}}}\right )\\
	&=\ln\left (1+\sqrt{\dfrac{1-x}{ax-x+1}}\right )-\frac{1}{2}\ln\left (1-\dfrac{1-x}{ax-x+1}\right )\\
	&=\ln\left (1+\sqrt{\dfrac{1-x}{ax-x+1}}\right )+\frac{1}{2}\ln\left (ax-x+1\right )-\frac{1}{2}\ln\left (ax\right ).
	\end{align*}
	The only terms in $F(x)$ that diverge in $x=0$ are $-\frac{d}{2}\ln(x)$ and $(-d)\left (-\frac{1}{2}\ln\left (ax\right )\right )$, which cancel (leaving the term $\frac{d}{2}\ln(a)$). Hence we have:
	$$F(0)=\frac{d-2}{2}\ln\left |\frac{4d-4}{d^2 a-4d+4}\right |
	-d \ln(2) +\frac{d}{2}\ln(a)
	+(d-2)\artanh\left (\frac{d-2}{d}\right ).$$
	Since $\artanh\left (\frac{d-2}{d}\right )=\frac{1}{2}\ln(d-1)$, we find in the limit $\varepsilon\to 0^+$:
	$$F(0) \underset{\varepsilon \to 0^+}{\to} 
	\frac{d-2}{2}\ln\left |\frac{4d-4}{d^2 4(d-1)\lambda-4d+4}\right |
	-d \ln(2) +\frac{d}{2}\ln(4(d-1)\lambda)
	+\frac{d-2}{2}\ln(d-1),
	$$
	and thus $\lim_{\varepsilon \to 0^+} I_{\lambda,\varepsilon} = \lim_{\varepsilon \to 0^+} F(1-\lambda\varepsilon)-F(0)$ is equal to
	$$ \lim_{\varepsilon \to 0^+} = \ln\left (
	\dfrac{\left (d^2 4(d-1)\lambda\right )^{\frac{d-2}{2}} }
	{ \left (4d-4\right )^{\frac{d-2}{2}} 2^{-d} \left ( 4(d-1)\lambda\right )^{\frac{d}{2}}
		\left (d-1\right )^{\frac{d-2}{2}} }
	\right )=\ln\left (\frac{d^{d-2}}{(d-1)^{d-1} \lambda}\right ),
	$$
	which concludes the proof.
\end{proof}

The methods used in Proposition \ref{prop:series:free} and Theorem \ref{thm:detFd} can be generalised to other operators and Cayley graphs, as in the following corollary.

\begin{coro}
	Let $d \geqslant 2$, and let $x_1, \ldots, x_{d}$ be  $d$ generators of the free group $\F_{d}$.
	Let $\zeta_1, \xi_1, \ldots,\zeta_{d},\xi_d \in  \C$ such that $|\zeta_1|=|\xi_1|=\ldots=|\zeta_{d}|=|\xi_d|=1$.
	Let $A= \zeta_1 R_{x_1} + \xi_1 R_{x_1^{-1}} +\ldots + \zeta_{d}  R_{x_{d}}+ \xi_d R_{x_d^{-1}}$.
	Then we have:
	\begin{enumerate}
		\item For $t$ small enough, the following generating series is equal to:
		$$u_{A^* A}(t) = \sum_{k=0}^{\infty} 
		\mathrm{tr}_{\F_d}\left ( (A^* A)^k
		\right ) t^k  = \dfrac{4d-2}{2d-2+2d\sqrt{1-4(2d-1)t}}.$$ 
		\item The Fugelede-Kadison determinant of $A$ is equal to
		$$
		\det{}_{\F_{d}}\left (\zeta_1 R_{x_1} + \xi_1 R_{x_1^{-1}} +\ldots + \zeta_{d}  R_{x_{d}}+ \xi_d R_{x_d^{-1}}\right ) = \dfrac{(2d-1)^{\frac{2d-1}{2}}}{(2d)^{d-1}}.
		$$
	\end{enumerate}
\end{coro}

\begin{proof}
	(1) The case where $\zeta_1=\xi_1=\ldots=\zeta_{d}=\xi_d=1$ follows from \cite{Ba, DL} (this time the circuits are considered in a $(2d)$-regular tree). 
	
	The general case follows from a similar argument as in the proof of Proposition \ref{prop:series:free}, with a slight difference: this time no letter is trivial, thus any trivial word must be of even length, and therefore any coefficient $\zeta_i$ (resp. $\xi_i$) is necessarily multiplied with a coefficient equal to $\zeta_i^*$ (resp. $\xi_i^*$), and vice-versa.
	
	(2) The result follows from (1) (i.e. the value of $u_{A^*A}(t)$) as in the proof of Theorem \ref{thm:detFd}, except that each $d$ is replaced with $2d$.
\end{proof}

\subsection{Consequences on Lehmer's constants}

Now, as a consequence of Theorem \ref{thm:detFd}, we obtain new upper bounds on Lehmer's constants and a negative answer to Question \ref{qu:lehmer:luck} (2) for a large class of  groups:

\begin{coro}\label{cor:lehmer}
	For every $d \geqslant 2$, Lehmer's constants
	$\Lambda(\F_d), \Lambda_1(\F_d), \Lambda^w(\F_d)$ and $\Lambda^w_1(\F_d)$ do not depend on $d$. 
	Moreover, for every $d \geqslant 2$, we have
	$$ \Lambda(\F_d) \leqslant \Lambda^w(\F_d) \leqslant \Lambda^w_1(\F_d) = \Lambda_1(\F_d) \leqslant \dfrac{2}{\sqrt{3}} = 1.1547... < \mathcal{M}(L) = 1.17628...$$

	In particular, any torsion-free group $G$ containing a subgroup $\F_d$ for  $d\geqslant 2$ (such as the fundamental group of a hyperbolic $3$-manifold, see \cite[C.3, C.26]{AFW})  also satisfies
	$$\Lambda(G), \Lambda_1(G), \Lambda^w(G), \Lambda^w_1(G) \in \left [1, \dfrac{2}{\sqrt{3}}\right ].$$
\end{coro}

\begin{rema}
	Note that groups containing nonabelian free groups are known as \textit{full-sized groups}, see for example \cite{BMRS}. The conclusion of Corollary \ref{cor:lehmer} thus applies to all torsionfree full-sized groups.
\end{rema}

\begin{proof}
	The first statement follows from Remark \ref{rem:lehmer:subgroup} and the fact every free group $\F_d$ (for $d\geqslant 2$)	injects in $\F_2$ and vice-versa.
	
	In the second statement, the first two inequalities follow immediately from the definitions, and the first equality from the fact that free groups satisfy the Strong Atiyah Conjecture (see Remark \ref{rem:atiyah} and \cite[Theorem 10.19]{Lu}). The third inequality follows from Theorem \ref{thm:detFd} and the first statement (observe that $d \mapsto \dfrac{(d-1)^{\frac{d-1}{2}}}{d^{\frac{d-2}{2}}}$ is increasing for $d\geqslant 3$, thus $\dfrac{2}{\sqrt{3}}$ is the best upper bound available).
	
	The third statement follows from the second one and Remark \ref{rem:lehmer:subgroup}: for any torsion-free group $G$ with $\F_d < G$, we have $\lambda(G) \leqslant \lambda(\F_d)$ for any $\lambda \in \{ \Lambda, \Lambda_1, \Lambda^w, \Lambda^w_1\}$.
\end{proof}

At the time of writing, $\dfrac{2}{\sqrt{3}} = 1.1547...$ appears to be the best known upper bound for $\Lambda_1^w(G)$ when $G$ is a generic torsion-free group with non-cyclic free subgroups.

There are specific subclasses of this large class of groups where we can go further and establish upper bounds 
of the form $\exp\left (\dfrac{\mathrm{vol}(M)}{6\pi}\right )$ that are even 
smaller than  $\dfrac{2}{\sqrt{3}} = 1.1547...$. We will discuss them in the next section (see Theorem \ref{thm:dehn:filling} and Corollary \ref{cor:dehn:filling}).

For generic torsion-free groups that do not have non-cyclic free subgroups (such as amenable groups), Question \ref{qu:lehmer:luck} (2) appears to be still open and $\mathcal{M}(L) = 1.17628...$ remains the best known upper bound for all Lehmer's constants. We refer to \cite{Lu2} for more detailed discussions on this subject.

\section{Lehmer's constants for groups of small hyperbolic manifolds}\label{sec:hyp}

When $M$ is a hyperbolic $3$-manifold, Lück's argument in \cite[Section 13]{Lu2} for the  case of closed $M$ also applies to the case of cusped $M$, and provides an upper bound on two Lehmer's constants of $\pi_1(M)$, as stated as follows:

\begin{prop}\label{prop:hyp:luck}
	Let $M$ be a hyperbolic $3$-manifold. Then
	$$\Lambda(\pi_1(M)) \leqslant \Lambda^w(\pi_1(M)) \leqslant \exp\left (\dfrac{\mathrm{vol}(M)}{6 \pi}\right ).$$
\end{prop}

\begin{proof}
	This follows from \cite[Section 13]{Lu2} and \cite[Theorem 4.9]{Lu}.
\end{proof}

We now show that for several hyperbolic $3$-manifolds, notably the ones with smallest volume, we can strenghten Proposition \ref{prop:hyp:luck} by bounding all four Lehmer's constants by the quantity $\exp\left (\dfrac{\mathrm{vol}(M)}{6 \pi}\right )$.

\begin{theo}\label{thm:dehn:filling}
	Let $M$ be a hyperbolic $3$-manifold, either closed or with cusps, and $G=\pi_1(M)$. Assume that $M$ satisfies one of the two following properties:
	\begin{enumerate}
		\item $M$ has cusps and $G$ has a presentation with two generators and one relator,
		\item $M$ is obtained by one or more Dehn fillings on a manifold $N$ satisfying (1).
	\end{enumerate}
	Then there exists an operator $A \in R_{\Z G}$ such that 
	$$\det{}_G(A)=T^{(2)}(M)=\exp\left (\dfrac{\mathrm{vol}(M)}{6 \pi}\right ).$$
	In particular, all four Lehmer's constants of $G$ admit the upper bound $\exp\left (\dfrac{\mathrm{vol}(M)}{6 \pi}\right ).$
\end{theo}

\begin{proof}
	If $M$ satisfies (1), then the result follows from  \cite[Theorem 4.9]{Lu}. 
	
	Now let us assume that $M$ satisfies (2), and let 
	$$N=N_0,  \ N_1, \ldots, N_{d-1}, \ N_d=M$$
	be a sequence of hyperbolic compact 3-manifolds with empty or toroidal boundary such that $N=N_0$ satisfies (1) and for $k=0, \ldots, d-1$, $N_{k+1}$ is obtained from $N_k$ by Dehn filling on a boundary component of $N_k$. 
	Let us denote $G_k:=\pi_1(N_k)$ (for $k=0, \ldots, d$) and $Q_k\colon G_k \twoheadrightarrow G_{k+1}$ (for $k=0, \ldots, d-1$) the group epimorphism induced by the Dehn filling process.
	
	Since $N$ satisfies (1), there exists an operator $A \in R_{\Z G_0}$ such that 
	$$\det{}_{G_0}(A)=T^{(2)}(N).$$
	
	Let us denote $c_1 \in G_1=\pi_1(N_1)$ the homotopy class of the core of the solid torus that was glued to $N_0$ by the Dehn filling process; one can also see $c_1$ as the class of a knot $K_1$ we can drill out from $N_1$ to obtain $N_0$.
	
	We remark that $c_1$ is non trivial. Indeed, if $c_1$ was trivial,  then one could choose the knot $K_1$ to be a small unknot lying in a $3$-ball; by drilling out $K_1$, the boundary of this $3$-ball would then be a separating sphere in $N_0$, which would contradict the fact that $N_0$ is hyperbolic (and thus irreducible).
	
	We now remark that $c_1$ has infinite order. This follows from the fact that $c_1$ is non-trivial, and that $G_1$ is torsion-free (because $N_1$ is hyperbolic, thus idrreducible with infinite fundamental group, see for instance \cite{AFW}).
	
	Finally, from the Dehn surgery formula for $L^2$-torsions (see Theorem \ref{thm:BA:surgery} or \cite[Proposition 4.2]{BA2}), since the core class $c_1$ has infinite order, we have
	$$T^{(2)}(N_1)=T^{(2)}(N_0,Q_0)=\det{}_{G_1}(Q_0(A)).$$
	
	The previous argument applies for any $k$, thus by  immediate induction, we have
	$$T^{(2)}(N_d)=T^{(2)}(N_{d-1},Q_{d-1})=
	\ldots = T^{(2)}(N_{0},Q_{d-1}\circ Q_{d-2} \circ \ldots \circ Q_0)=
	\det{}_{G_d}(A'),$$
	where $A'= \left (Q_{d-1}\circ Q_{d-2} \circ \ldots \circ Q_0\right )(A) \in R_{\Z G_d}$.
	
	The conclusion follows from the Lück--Schick theorem (see Theorem \ref{thm:LS}).
\end{proof}

\begin{coro}\label{cor:dehn:filling}
	Let $M$ be a hyperbolic $3$-manifold obtained by Dehn fillings on the exterior of the Whitehead link. Then
	$$\Lambda^w_1(\pi_1(M)) \leqslant 
	\exp\left (\dfrac{\mathrm{vol}(M)}{6 \pi}\right ).$$
	In particular, this applies to the Weeks manifold $M_{We}$.
\end{coro}

\begin{proof}
	This follows from Theorem \ref{thm:dehn:filling} and the fact that the group of the Whitehead link admits for instance the presentation (by simplifying a Wirtinger presentation)
	$$\langle a,b| [a,[a,b][a,b^{-1}]] \rangle,$$
	where we denote $[g,h]:=g h g^{-1} h^{-1}$ the commutator of $g,h$, and $a, b$ are the classes of meridians of each of the two components of the Whitehead link.
\end{proof}

There is a finite number of cusped hyperbolic 3-manifolds with volume less than $6 \pi \ln\left (\frac{2}{\sqrt{3}}\right )=2.711...$ (see for instance \cite{GMM}).
For each such cusped manifold $M$, the group $\pi_1(M)$ admits a presentation with two generators and one relator (this can be checked for each case on the software \textit{Snappy} \cite{CDGW}), thus it follows from Theorem \ref{thm:dehn:filling} that
$$\Lambda^w_1(\pi_1(M)) \leqslant 
\exp\left (\dfrac{\mathrm{vol}(M)}{6 \pi}\right ) <
\frac{2}{\sqrt{3}}=1.1547...$$
for each such $M$.

The same condition holds for the infinite number of closed hyperbolic 3-manifolds with volume less than $6 \pi \ln\left (\frac{2}{\sqrt{3}}\right )=2.711...$ and satisfying the conditions of Theorem \ref{thm:dehn:filling} or Corollary \ref{cor:dehn:filling} (the one with smallest volume being the Weeks manifold). 

For fundamental groups of hyperbolic 3-manifolds not covered by the previous two paragraphs, the upper bound
$$\Lambda^w_1(\pi_1(M)) \leqslant 
\Lambda^w_1(\F_2) \leqslant
\frac{2}{\sqrt{3}}=1.1547...$$
(which follows from Corollary  \ref{cor:lehmer}) remains the best one at the moment.

\section{Approximate computations of Fuglede-Kadison determinants}\label{sec:approx}

This section builds upon \cite[Section 3.7]{Lu} and \cite[Appendix B]{BA}.

Let $G$ be a finitely presented group and let $A \in R_{\C G}$ denote an injective right multiplication operator on $\ell^2(G)$ by a non-zero element of the group algebra. In general, trying to compute the exact value of $\det_G(A)$ with techniques such as Proposition \ref{prop:det_traces} may prove too difficult, either the step of computing the generating series, or the value of the integral, or both. 

However, in several cases, an \textit{approximate} value of $\det_G(A)$ can still yield useful information. For instance, if $A \in R_{\Z G}$, then an upper bound $D>1$ on $\det_G(A)$ will provide the same upper bound $D$ for Lehmer's constant $\Lambda_1^w(G)$. This motivates us to mention a general method of obtaining a sequence of upper approximations for $\det_G(A)$, once again via combinatorial group theory.

\begin{prop}\label{prop:upper}
	Let $G$ be a finitely presented group and let $A \in R_{\C G}$ denote an injective right multiplication operator on $\ell^2(G)$ by a non-zero element of the group algebra. Then for any fixed $\lambda \in \left (0,\|A\|^{-2}\right )$, the sequence 
	$$	\sqrt{\dfrac{1}{\lambda}\exp  \left (-\sum_{n=1}^{N} \frac{1}{n}\mathrm{tr}_{G}\left (
		\left (
		\Id - \lambda A^* A
		\right )^n
		\right )\right )}
	$$
	decreases towards $\det_G(A)$ as $N \to \infty$.
\end{prop}

\begin{proof}
	From the proof of Proposition \ref{prop:det_traces},  for any fixed $\lambda \in \left (0,\|A\|^{-2}\right )$, we have	
	$$\det{}_{G}(A) = \lim_{\varepsilon \to 0^+}
	\sqrt{\dfrac{1}{\lambda}\exp  \left (-\sum_{n=1}^{\infty} \frac{1}{n}\mathrm{tr}_{G}\left (
		\left (
		(1-\lambda \varepsilon) \Id - \lambda A^* A
		\right )^n
		\right )\right )}.$$
	Since $A^*A$ is positive and $\mathrm{tr}_G$ is monotone, we have for all $n\geqslant 1$ and $\varepsilon<\varepsilon'$ that
	$$0 \leqslant \mathrm{tr}_{G}\left (
	\left (
	(1-\lambda \varepsilon') \Id - \lambda A^* A
	\right )^n
	\right ) \leqslant
	\mathrm{tr}_{G}\left (
	\left (
	(1-\lambda \varepsilon) \Id - \lambda A^* A
	\right )^n
	\right ).$$
	Thus, by taking partial sums from $1$ to any fixed $N$, we have:
	\begin{align*}
	\det{}_G(A) 
	&\leqslant 	\sqrt{\dfrac{1}{\lambda}\exp  \left (-\sum_{n=1}^{N} \frac{1}{n}\mathrm{tr}_{G}\left (
		\left (
		(1-\lambda \varepsilon) \Id - \lambda A^* A
		\right )^n
		\right )\right )}\\
	&\leqslant 	\sqrt{\dfrac{1}{\lambda}\exp  \left (-\sum_{n=1}^{N} \frac{1}{n}\mathrm{tr}_{G}\left (
		\left (
		(1-\lambda \varepsilon') \Id - \lambda A^* A
		\right )^n
		\right )\right )},
	\end{align*}
	where the first inequality becomes an equality as $N\to \infty$ and $\epsilon \to 0^+$.
	
	From what precedes, we get the following upper bound
	\begin{align*}
	\det{}_G(A) 
	&\leqslant 	\sqrt{\dfrac{1}{\lambda}\exp  \left (-\sum_{n=1}^{N'} \frac{1}{n}\mathrm{tr}_{G}\left (
		\left (
		\Id - \lambda A^* A
		\right )^n
		\right )\right )}\\
	&\leqslant 	\sqrt{\dfrac{1}{\lambda}\exp  \left (-\sum_{n=1}^{N} \frac{1}{n}\mathrm{tr}_{G}\left (
		\left (
		\Id - \lambda A^* A
		\right )^n
		\right )\right )},
	\end{align*}
	for all $N' > N \geqslant 1$. The result follows.
\end{proof}

We can now use Proposition \ref{prop:upper} to obtain upper approximations of $\det_G(A)$. For this, we also need a \textit{method for solving the word problem in $G$}. Once we have such a method, with enough time we can compute  any finite number of terms
$\mathrm{tr}_{G}\left (
\left (
\Id - \lambda A^* A
\right )^n
\right )$ and thus a partial sum up to $N$.

We illustrate such a process with the example of the fundamental group $G=\pi_1(m004)$ of the figure-eight knot complement $m004$. Remark that $m004$ is one of the two cusped hyperbolic manifolds with smallest volume (equal to $\mathrm{vol}(m004)=2.0298...$), and it follows from Section \ref{sec:hyp}  that 
$$\Lambda^w_1(G) \leqslant T^{(2)}(m004) = \det{}_{G}(A)=
\exp\left (\dfrac{\mathrm{vol}(m004)}{6 \pi}\right ) =1.113...$$
for several possible operators $A\in R_{\Z G}$.

A convenient way to solve the word problem in a linear group $H$ is to use a linear \textit{faithful} representation $\rho\colon H \hookrightarrow GL_2(\C)$ and to compare a matrix $\rho(g)$ (for $g\in H$) to the matrix $\begin{pmatrix}
1 & 0 \\ 0 & 1
\end{pmatrix}$. When $M$ is a hyperbolic 3-manifold and $H=\pi_1(M)$, there exists such a faithful representation $\rho$, which is built from the complete hyperbolic structure on $M$; the computer program \textit{Snappy} \cite{CDGW} computes the value of this representation $\rho$ on the generators of $H$. We  use this method for $H=G=\pi_1(m004)$ for two different operators.

From a diagram of the figure-eight knot, we can obtain a Wirtinger presentation of $G$, which reduces to the presentation
$$P = \langle x,y | x y x^{-1} y x = y x y^{-1} x y \rangle. $$
Remark that the abelianization goes from $G$ to $\Z$ and sends $x,y$ to $1$. We consider
the operator 
$$A_t := Id - t R_y - t R_{x y x^{-1}} - t R_{y x y^{-1}} + t^2 R_{x y x^{-1} y},$$
defined for all $t>0$. As a function of $t>0$, the Fuglede--Kadison determinant of $A_t$ gives the $L^2$-Alexander invariant of the figure-eight knot $4_1$:
$$\Delta^{(2)}_{4_1}(t) = \det{}_{\NN(G_{K})}(A_t).$$
When $t=1$, we have $\det{}_{G}(A_1)=T^{(2)}(m004) =
\exp\left (\frac{\mathrm{vol}(m004)}{6 \pi}\right ) =1.113...$

Remark that since the figure-eight knot is fibered, $G$ is a semi-direct product of $\F_2$ and $\Z$ and thus satisfies the strong Atiyah conjecture.

We compute upper approximations of $\Delta^{(2)}_{4_1}(t)$ from the operator $A_t$. For each $t>0$, we choose $\lambda^{-1} = 1 + 3 t + t^2$ on the interval $I = [0.001;4]$. The approximations up to $N=6$ are drawn in blue in Figure \ref{figure sage 41}. The known exact values of $\Delta^{(2)}_{4_1}(t)$ (which are $1$ on $(0;0.38)$, $1.113...$ in $1$ and $t^2$ on $(2.618;4)$, see \cite{BA}) are drawn in red.

\begin{figure}[!h]
	\centering
	\includegraphics[scale=0.65]{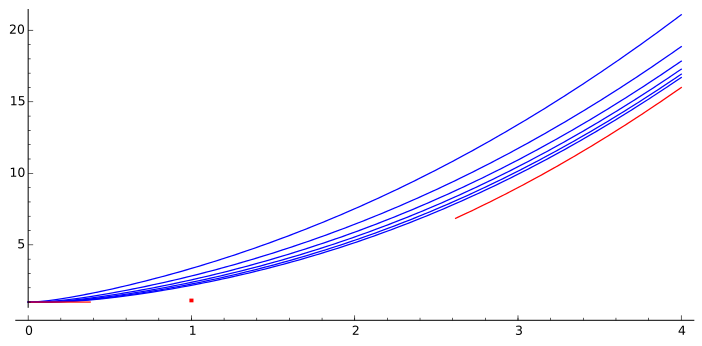}
	\caption{Upper approximation for $\Delta_{4_1}^{(2)}(t)$ from the Wirtinger presentation.} \label{figure sage 41}
\end{figure}

Since the figure-eight knot complement is obtained from the one of the Whitehead link via a Dehn filling of slope $-1$, the  group $G$ admits the presentation
$$ P'= \langle a_1, \alpha, \beta |  [a_1, \alpha ] [a_1^{-1}, \alpha] = \beta, \alpha \beta^{-1} = 1 \rangle.$$
Here the abelianization to $\Z$ acts as:
$a_1 \mapsto 1, \ \alpha, \beta \mapsto 0.$
We consider the operator 
$$A'_t = Id - R_{a_1 \alpha a_1^{-1}} - \dfrac{1}{t}R_{[a_1,\alpha]a_1^{-1}} + \dfrac{1}{t}R_{[a_1,\alpha]a_1^{-1}\alpha},$$
and the $L^2$-Alexander invariant is here equal to
$$\Delta^{(2)}_{4_1}(t)= \det{}_{\NN(G_K)}(tA'_t) \cdot \max(1,t).$$
Once again, for $t=1$, we have $\det{}_{G}(A'_1)=T^{(2)}(m004) =
\exp\left (\frac{\mathrm{vol}(m004)}{6 \pi}\right ) =1.113...$

We compute upper approximations of $\Delta^{(2)}_{4_1}(t)$ from approximations $f_N(t)$ of $\det_{\NN(G_K)}(t A'_t)$. We choose $\lambda^{-1} = 2 + 2 t$ and the interval $I = [0.001;4]$. The approximations
$\left (f_N(t) \cdot \max(1,t)\right )$
up to $N=7$ are drawn in blue in Figure \ref{figure sage 41 twist}. 
The known exact values of $\Delta^{(2)}_{4_1}(t)$ (which are $1$ on $(0;0.38)$, $1.113$ in $1$ and $t^2$ on $(2.618;4)$) are drawn in red.

\begin{figure}[!h]
	\centering
	\includegraphics[scale=0.65]{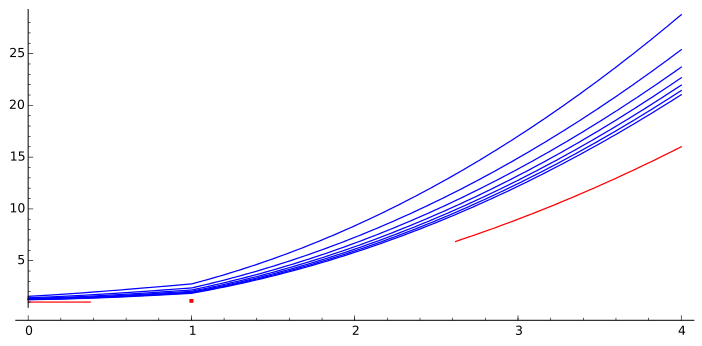}
	\caption{Upper approximation for $\Delta_{4_1}^{(2)}(t)$ from the twist knot presentation.} \label{figure sage 41 twist}
\end{figure}

The forms of the upper approximations in
Figures \ref{figure sage 41} and \ref{figure sage 41 twist} led us to conjecture  that the $L^2$-Alexander invariant was always a convex function \cite[Conjecture 6.1]{BA}. Liu later proved in \cite{Liu} that this invariant is always a \textit{multiplicatively convex} function,  which implies in particular that it is indeed convex and continuous.
This illustrates the potential of computing such upper approximations.

\section*{Acknowledgements}
The author was supported by the FNRS in his "Research Fellow" position at UCLouvain, under under Grant
no. 1B03320F. We thank Cristina Anghel-Palmer, Anthony Conway, Jacques Darné and Wolfgang Lück for helpful discussions,  \'Elie de Panafieu for his help in writing the SAGE algorithm used in Section \ref{sec:approx}, and the anonymous referees for helpful comments and suggestions.


\begin{thebibliography}{10}





	\bibitem{AFW} M. Aschenbrenner, S. Friedl and H. Wilton. 
	\emph{$3$-manifold groups},
EMS Series of Lectures in Mathematics. European Mathematical Society (EMS),  xiv+215 pp., Zürich, 2015.
 
\bibitem{Ba} L. Bartholdi.
Counting paths in graphs,
\emph{Enseign. Math.} (2) 45, no. 1-2, 83--131, 1999.
 
\bibitem{BA} F.~Ben Aribi. 
\emph{A study of properties and computation techniques of the $L^2$-Alexander
	invariant in knot theory}, PhD thesis, Université Paris Diderot, Paris, 2015.
 
\bibitem{BA2} F.~Ben Aribi. 
{Gluing formulas for the $L^2$-Alexander torsions},
\emph{Commun. Contemp. Math.}
 21, no. 3, 1850013, 31 pp., 2019.
 
\bibitem{BAC} F.~Ben Aribi and A.~Conway. 
{$L^2$-Burau maps and $L^2$-Alexander torsions},
\emph{Osaka J. Math.} 55,
529--545, 2018.
 
 	  
\bibitem{Bo}
D. Boyd. Speculations concerning the range of Mahler's measure,  
\emph{Canad. Math. Bull.} 24, no. 4, 453--469, 1981.
 	 
\bibitem{BMRS}
M. R. Bridson, D. B. McReynolds, A. W. Reid, R. Spitler. 
 Absolute profinite rigidity and hyperbolic	geometry, 
 \emph{Ann. of Math.}
  (2) 192, no. 3, 679--719, 2020.
 
\bibitem{Bu}
W. Burau. 
Über Zopfgruppen und gleichsinnig verdrillte Verkettungen, 
\emph{Abh. Math. Sem. Univ. Hamburg}
11, 179--186, 1935.
 
\bibitem{CDGW}
Marc Culler, Nathan M. Dunfield, Matthias Goerner, and Jeffrey R. Weeks. 
\emph{SnapPy, a computer	program for studying the geometry and topology of 3-manifolds}, 
available at \url{http://snappy.
	computop.org}
 
\bibitem{DL}
O. Dasbach and M. Lalin.
	Mahler measure under variations of the base group,
\emph{Forum Math.}
 21, no. 4, 621--637, 2009.

\bibitem{DFL}
J. Dubois, S. Friedl and W. L\"uck. 
	The $L^2$-Alexander torsion of 3-manifolds,
\emph{Journal of Topology}
 9, No. 3, 889--926, 2016.
 	 
\bibitem{Fox}
R.H. Fox. {Free differential calculus. II. The isomorphism problem of groups}, 
\emph{Ann. of Math.} (2), 59, 196--210, 1954.
 
\bibitem{GMM}
D. Gabai, R. Meyerhoff and P. Milley. 
{Minimum volume cusped hyperbolic three-manifolds},
\emph{J. Amer. Math. Soc.} 22(4), 1157--1215, 2009.
 
\bibitem{KW} A. Kricker and Z. Wong. 
{Random Walks on Graphs and Approximation of $L^2$-Invariants}, 
\emph{Acta Mathematica Vietnamica}, Volume 46, Issue 2, pp. 309--319, March 2021.
 
\bibitem{LZ} W. Li and W. Zhang. 
{An $L^2$-Alexander invariant for knots}, 
\emph{Commun. Contemp. Math.}
 \textbf{8}, no. 2, 167--187, 2006.
 
\bibitem{Liu}
Y. Liu. 
{Degree of $L^2$-Alexander torsion for 3-manifolds}, 
\emph{Invent. Math.}
 207, 981--1030, 2017.
 
\bibitem{Lu} 
W. L\"uck. 
\emph{$L^2$-invariants: theory and applications to geometry and $K$-theory},
 Ergebnisse der Mathematik und ihrer Grenzgebiete. 3. Folge. A Series of Modern Surveys in Mathematics, 44. Springer-Verlag, Berlin, 2002.
 
\bibitem{Lu2} W. L\"uck. 
{Lehmer's Problem for arbitrary groups}, to appear in \emph{Journal of Topology and Analysis}, arXiv:1901.00827, 2019.
 
\bibitem{LS}
W. Lück and T. Schick. 
{$L^2$-torsion of hyperbolic manifolds of finite volume}, 
\emph{Geom. Funct. Anal.}
 9, no. 3, 518--567, 1999.
 
\bibitem{W}
Wolfram Research, Inc., Wolfram|Alpha Notebook Edition, Champaign, IL (2022), \url{https://www.wolframalpha.com/}.






\end{thebibliography}
\end{document}